\documentclass[12pt]{article}

\usepackage{amsmath,amsfonts,amssymb,latexsym,amsthm,verbatim,a4wide,graphicx,hyperref}

 \theoremstyle{plain}
\newtheorem{theorem}{Theorem}[section]
\newtheorem{lemma}[theorem]{Lemma}

\newtheorem{definition}[theorem]{Definition}
\newtheorem{example}[theorem]{Example}
\newtheorem{condition}{Condition}

\newtheorem{remark}[theorem]{Remark}

\newcommand{\ep}{{\mathbb E}}
\newcommand{\pr}{{\mathbb P}}
\newcommand{\re}{{\mathbb R}}
\newcommand{\Z}{{\mathbb Z}}
\newcommand{\vc}[1]{{\mathbf{#1}}}

\newcommand{\bin}{{\rm Bin}}
\newcommand{\var}{{\rm Var}}
\newcommand{\C}[1]{{\bf ULC}(#1)}
\newcommand{\Ent}[2]{{\rm Ent}_{#1}(#2)}

\newcommand{\Gam}[1]{\Gamma_{#1}}
\newcommand{\tends}{\rightarrow \infty}
\newcommand{\EE}{{\mathcal{E}}_P}

\newcommand{\bi}{\begin{itemize}}

\newcommand{\ei}{\end{itemize}}

\begin{document}

\title{Entropy and thinning of discrete random variables}
\date{\today}
\author{Oliver Johnson\thanks{School of Mathematics, Univeristy of Bristol, University Walk, Bristol, BS8 1TW, UK. Email: {\tt maotj@bristol.ac.uk}}}
\maketitle

\begin{abstract} \noindent
We describe five types of results concerning information and concentration of discrete random variables, and relationships
between them, motivated by their counterparts in the continuous case. The results we consider are
information theoretic approaches to Poisson approximation, the maximum entropy property of the Poisson distribution, 
discrete concentration (Poincar\'{e} and logarithmic Sobolev) inequalities, monotonicity of entropy and concavity of entropy
in the Shepp--Olkin regime.
\end{abstract}

\section{Results in continuous case} \label{sec:conts}
This paper  gives a personal review of a number of  results concerning the entropy and concentration properties of discrete random variables. 
For simplicity, we  only consider independent  random variables  (though it is an extremely interesting open problem to extend many of the results
to the dependent case).
These results are generally motivated by their
counterparts in the continuous case, which we will briefly review, using notation which holds only for Section \ref{sec:conts}.

For simplicity we restrict our attention in this section to random variables taking values in $\re$.
 For any probability density $p$, write
 $\lambda_p = \int_{-\infty}^{\infty}x p(x) dx$ for its mean and $\var_p = \int_{-\infty}^{\infty} (x - \lambda_p)^2 p(x)
dx$ for its variance.  We write $h(p)$ for the differential entropy of  $p$, and interchangeably write
$h(X)$ for $X \sim p$. Similarly we write $D(p \| q)$ or $D(X \| Y)$ for relative entropy.
We write $\phi_{\mu,\sigma^2}(x)$ for
the density of Gaussian $Z_{\mu,\sigma^2} \sim N(\mu, \sigma^2)$.

Given a function $f$, we wish to measure its concentration properties with respect to probability 
density $p$; we write $\lambda_{p}(f) = \int_{-\infty}^{\infty}f(x) p(x) dx$ for the expectation of 
$f$ with respect to $p$, write $\var_p(f) = \int_{-\infty}^{\infty}p(x) (f(x) - \lambda_{p}(f))^2 dx$ for the variance
and define 
\begin{equation} \label{eq:defentcont}
\Ent{p}{f} = \int_{-\infty}^{\infty} p(x) f(x) \log f(x) dx - \lambda_{p}(f) \log \lambda_{p}(f).\end{equation}
We briefly summarize the five types of results we  study in this paper, the discrete analogues of which are
described in more detail in the five sections from Sections \ref{sec:score} to \ref{sec:so} respectively.
\begin{enumerate}
\item{{\bf [Normal approximation]}}
A paper by Linnik \cite{linnik} was the first to attempt to use ideas from information theory to prove the Central Limit Theorem, with
later contributions by Derriennic \cite{derriennic} and Shimizu \cite{shimizu}.
This idea was developed considerably by Barron \cite{barron}, building on results of Brown \cite{brown}. Barron  considered 
 independent and identically distributed $X_i$ with mean $0$ and variance $\sigma^2$. He proved \cite{barron}
that relative entropy $D( (X_1 + \ldots + X_n)/\sqrt{n} \| Z_{0,\sigma^2})$ converges to 
zero if and only if it ever becomes finite.
Carlen and Soffer \cite{carlen}  proved similar results for non-identical variables under a Lindeberg-type condition.

While none of these papers gave an explicit rate of convergence, later work \cite{artstein2,ball,johnson5} proved
that relative entropy converges at an essentially optimal $O(1/n)$ rate, under the (too strong)
assumption of finite Poincar\'{e} constant (see
Definition \ref{def:poincont} below). Typically (see  \cite{johnson14} for a review), these papers did not manipulate entropy directly, but rather  used  properties of projection operators to control
the behaviour of Fisher information on convolution. The use of Fisher information is based on the de Bruijn identity (see \cite{barron,blachman,stam}),
which relates entropy to Fisher information, using the fact that (see for example \cite[Equation (5.1)]{stam}
 for fixed random variable $X$, the density $p_t$ of $X + Z_{0,t}$ satisfies a heat equation of the form
\begin{equation} \label{eq:heateqn}
\frac{\partial p_t}{\partial t}(z) = \frac{1}{2} \frac{\partial^2 p_t}{\partial z^2}(z) = \frac{1}{2} \frac{\partial}{\partial z} \left( p_t(z)  \rho_t(z) \right),\end{equation}
where $\rho_t(z) := \frac{\partial p_t}{\partial z}(z)/p_t(z)$ is the score function with respect to location parameter, in the Fisher sense.

More recent work of Bobkov, Chistyakov and 
G\"{o}tze \cite{bobkov6,bobkov8} used properties of characteristic functions to remove the assumption of finite
Poincar\'{e} constant, and even extended this theory to include convergence to other stable laws \cite{bobkov7,bobkov10}.
 
\item{{ \bf [Maximum entropy]}} In \cite[Section 20]{shannon2}, Shannon described maximum entropy results for 
continuous random variables in exponential family form.  That is, given a test function $\psi$, positivity of the relative
entropy shows that the entropy $h(p)$ is maximised for fixed values of $\int_{-\infty}^{\infty}p(x) \psi(x) dx$ by densities proportional to
$\exp(- c \psi(x))$ for some value of $c$.

This is a very useful result in many cases, and gives us well-known facts such as that entropy is maximised under a variance
constraint by the Gaussian density. This motivates the  information theoretic approaches to normal approximation
described above, as well as the monotonicity of entropy result of \cite{artstein} described later.

However, many interesting densities cannot be expressed in exponential family form for a natural choice of $\psi$, and we  
would like to  extend maximum entropy theory to cover these.
As remarked in \cite[p.215]{gnedenko2}, ``immediate candidates to be subjected to such analysis are, of course, stable laws''.
For example, consider the Cauchy density, for which it is not at all natural to
consider the expectation of $\psi(x) = \log(1 + x^2)$ in order to define a maximum entropy class. In the discrete context
of Section \ref{sec:maxent} we discuss the Poisson mass function, again believing that 
$\ep \log X!$ is not a natural object of study.

\item{{\bf [Concentration inequalities]}}
Next we consider concentration inequalities; that is, relationships between senses in which a function is approximately constant.
We focus on Poincar\'{e} and logarithmic Sobolev inequalities.
\begin{definition} \label{def:poincont}
We say that a probability density $p$ has Poincar\'{e} constant $R_p$, if for all sufficiently well-behaved functions $g$,
\begin{equation} \label{eq:poincont} \var_p(g) \leq R_p \int_{-\infty}^{\infty}p(x) g'(x)^2 dx. \end{equation}
\end{definition}
Papers vary in the exact sense of `well-behaved' in this definition, and for the sake of brevity we will not focus on this issue;
we will simply suppose that both sides of \eqref{eq:poincont} are well-defined.
\begin{example} \label{ex:chernoff}
Chernoff \cite{chernoff} (see also \cite{borovkov2}) 
used an expansion in Hermite polynomials (see \cite{szego}) to show that the Gaussian densities
$\phi_{\mu,\sigma^2}$ have Poincar\'{e} constant equal to $\sigma^2$. By considering $g(t) = t$, it is clear that for
any random variable $X$ with finite variance, $R_p \geq \var_p$. Chernoff's result 
\cite{chernoff} played a central role in the information theoretic proofs of the Central
Limit Theorem of  Brown \cite{brown} and Barron \cite{barron}. 
\end{example}
In a similar way, we define the log-Sobolev constant, in the form discussed in \cite{bakry2}:
\begin{definition} \label{def:lsicont}
We say that a probability density $p$ satisfies the 
 logarithmic Sobolev inequality with constant
$C$ if for any function $f$ with positive values:
\begin{equation} \label{eq:lsicont} \Ent{p}{f} \leq \frac{C}{2} \int_{-\infty}^{\infty} \frac{\Gam{1}(f,f)(x)}{f(x)} p (x) dx, \end{equation}
where we write $\Gam{1}(f,g) = f' g'$, and view the RHS as a Fisher-type expression.
\end{definition}
This is an equivalent formulation  of the log-Sobolev inequality first stated by Gross \cite{gross}, who proved that the
Gaussian density $\phi_{\mu,\sigma^2}$ has log-Sobolev constant $\sigma^2$. In fact, Gross's 
Gaussian log-Sobolev inequality \cite{gross} follows from
Shannon's Entropy Power Inequality \cite[Theorem 15]{shannon2}, as proved by the Stam--Blachman approach \cite{blachman,stam}.

However, we would like to find results that strengthen Chernoff and Gross's results, to provide inequalities of 
similar functional form, under assumptions that $p$ belongs to particular classes of densities.
Gross's Gaussian log-Sobolev inequality can be considerably generalized by the 
Bakry--\'{E}mery calculus \cite{bakry} (see \cite{ane,bakry2,guionnet} for reviews of this theory).
Assuming that the so-called Bakry-\'{E}mery condition is satisfied, 
 taking $U(t) = t^2$ in \cite[Proposition 5]{bakry} we deduce (see also 
\cite[Proposition 4.8.1]{bakry2}):
\begin{theorem} \label{thm:contpoin} If 
the Bakry-\'{E}mery condition holds with constant $c$ then the Poincar\'{e} inequality holds with constant
$1/c$.
\end{theorem}
Similarly \cite[Theorem 1]{bakry} (see also \cite[Proposition 5.7.1]{bakry2}) gives that:
\begin{theorem} \label{thm:contlsi}
If 
the Bakry-\'{E}mery condition 
 holds with constant $c$ then the logarithmic Sobolev inequality holds with constant
$1/c$.
\end{theorem}
If $p$ is Gaussian with variance $\sigma^2$, since (see \cite[Example 4.18]{guionnet}), the Bakry-\'{E}mery condition holds
with constant $c=1/\sigma^2$,  we 
recover the original Poincar\'{e} inequality of Chernoff \cite{chernoff} and the
log-Sobolev inequality of Gross \cite{gross}. Indeed, if the ratio $p/\phi_{\mu,\sigma^2}$ is a log-concave function, then
this approcah shows that
the Poincare and log-Sobolev inequalities hold with constant $\sigma^2$.

\item{{\bf [Monotonicity of entropy]}} While  Brown \cite{brown} and Barron \cite{barron} exploited ideas of Stam \cite{stam}
to deduce that entropy of $(X_1 + \ldots + X_n)/\sqrt{n}$
increases in the Central Limit Theorem regime along the `powers of 2 subsequence' $n=2^k$, it remained an
open problem for many years to show that the entropy is monotonically increasing in all $n$. This conjecture may well date back to
Shannon, and was  mentioned  by Lieb \cite{lieb2}.
It was resolved by 
Artstein, Ball, Barthe and Naor \cite{artstein}, who proved a more general result using a variational characterization of 
Fisher information (introduced in \cite{ball}):
\begin{theorem}[\cite{artstein}] \label{thm:hmonotone}
Given independent continuous  $X_i$ with
finite variance, for any positive $\alpha_i$ such that 
$\sum_{i=1}^{n+1} \alpha_i =1$,
 writing $\alpha^{(j)}  = 1-\alpha_j$, then
\begin{equation} \label{eq:hmonotone}
n h \left( \sum_{i=1}^{n+1} \sqrt{\alpha_i} X_i \right) \geq \sum_{j=1}^{n+1}
\alpha^{(j)} h\left( \sum_{i \neq j} \sqrt{\alpha_i/\alpha^{(j)}} X_i \right)
.\end{equation}
\end{theorem}
This result is referred to as monotonicity because,  choosing $\alpha_i = 1/(n+1)$ for independent and identically distributed (IID) $X_i$, \eqref{eq:hmonotone} shows that
 $h \left( \sum_{i=1}^{n} X_i/\sqrt{n} \right)$ is monotonically increasing
in $n$.  Equivalently,  relative entropy 
$D \left( \sum_{i=1}^{n} X_i/\sqrt{n} \right \| Z)$ is  monotonically decreasing
in $n$.
This means that the Central Limit Theorem can be viewed as an equivalent of the Second Law of Thermodynamics.
An alternative proof of this monotonicity result is given by Tulino and Verd\'{u} \cite{tulino}, based on the MMSE characterization
of mutual information of \cite{guo}.

The form of \eqref{eq:hmonotone} in the case $n=2$ was previously a well-known result; indeed (see for example
\cite{dembo}) it is equivalent to  Shannon's Entropy Power Inequality \cite[Theorem 15]{shannon2}.
However, 
Theorem \ref{thm:hmonotone}
 allows a stronger form of the  Entropy Power Inequality to be proved, 
referring to `leave-one-out' sums (see \cite[Theorem 3]{artstein}. This was generalized by Madiman and Barron
\cite{madiman}, who considered the entropy power of sums of arbitrary subsets of the original variables.
The more recent paper \cite{madiman6} reviews in detail relationships between the Entropy Power Inequality and results in convex geometry,
including the Brunn-Minkowski inequality.

\item{{\bf [Concavity of entropy]}}
In the setting of probability measures taking values on the real line $\re$, there has been considerable interest in optimal
transportation of measures; that is, given densities $f_0$ and $f_1$, to find a sequence of densities $f_t$ smoothly interpolating
between them in a way which minimises some appropriate cost function. For example, using the natural quadratic cost function
induces the Wasserstein distance $W_2$ between $f_0$ and $f_1$.
 This theory is extensively reviewed in the two
books by Villani \cite{villani3,villani}, even for probability measures taking values on a Riemannian manifold.
In this setting, work by authors including Lott, Sturm and Villani \cite{lott,sturm,sturm2}  relates the concavity of the entropy 
$h(f_t)$ as a function of $t$ to the Ricci curvature of the underlying manifold.

Key works in this setting include those by Benamou and Brenier \cite{benamou2,benamou} (see also \cite{carlen4}), who gave a 
variational characterization of the Wasserstein distance between $f_0$ and $f_1$. Authors such as
 Caffarelli \cite{caffarelli} and
Cordero-Erausquin \cite{cordero3} used these ideas to give new proofs of results such as  log-Sobolev 
and transport inequalities.
Concavity of entropy also plays a key role in the field of information geometry (see \cite{amari4,amari,rao}), where the Hessian of
entropy  induces a metric on the space of probability measures.
\end{enumerate}

\section{Technical definitions in the discrete case}

In the remainder of this paper, we describe results which can be seen as discrete analogues of the results of Section \ref{sec:conts}. In that section, a distinguished role
was played by the set of  Gaussian densities, a set which has attractive properties including closure under convolution and scaling, and an explicit value for their  entropies.
We argue that a similar role is played by the Poisson mass functions, a class which is preserved on convolution and thinning (see Definition \ref{def:thin} below),
 even though (see \cite{adell2}) we only have entropy bounds.

We now make some definitions which will be useful throughout the remainder of the paper. From now, we will
assume all random variables take values on the positive integers $\Z_+$. 
Given a probability mass function (distribution) $P$, we write
$\lambda_P = \sum_{x=0}^{\infty}x P(x)$ for its mean and $\var_P = \sum_{x=0}^{\infty} (x- \lambda_P)^2 P(x)$ for its
variance. 
In particular, write $\Pi_\lambda(x) = e^{-\lambda} \lambda^x/x!$ for
the mass function of a Poisson random variable with mean $\lambda$ and $B_{n,p}(x) = \binom{n}{x} p^x (1-p)^{n-x}$ for the mass function of a Binomial $\bin(n,p)$. We write $H(P)$ for the discrete entropy of a probability distribution $P$,
and interchangably write
$h(X)$ for $X \sim P$.
For any random variable $X$ we
write $P_X$, defined by $P_X(x) = \pr(X = x)$, for its  probability mass function.

Given a function $f$ and mass function $P$, we write $\lambda_{P}(f) = \sum_{x=0}^{\infty}f(x) P(x)$ for the expectation of 
$f$ with respect to $P$, write $\var_P(f) = \sum_{x=0}^{\infty}P(x) (f(x) - \lambda_{P}(f))^2$
and define 
\begin{equation} \label{eq:defent}
\Ent{P}{f} = \sum_{x=0}^{\infty}P(x) f(x) \log f(x) - \lambda_{P}(f) \log \lambda_{P}(f).\end{equation}
Note that if $f = Q/P$, where $Q$ is a probability mass function, then 
\begin{equation} \Ent{P}{f} = \sum_{x=0}^{\infty}Q(x) \log \left( \frac{Q(x)}{P(x)} \right) = D(Q \| P),\end{equation}
the relative entropy from $Q$ to $P$.
We define the operators $\Delta$ and $\Delta^*$ (which are adjoint with respect to counting measure) by
$\Delta f(x) = f(x+1) - f(x)$
and $\Delta^* f(x) = f(x-1) - f(x)$, with the convention that $f(-1) = 0$.

We now make two further key definitions, in Condition \ref{cond:ulc} we define
 the ultra-log-concave (ULC) random variables (as introduced by Pemantle
\cite{pemantle} and Liggett \cite{liggett}) and in Definition \ref{def:thin} we define the thinning operation introduced by 
R\'{e}nyi \cite{renyi4} for point processes.

\begin{condition}[ULC] \label{cond:ulc}
For any $\lambda$, define  the class of  ultra-log-concave (ULC) mass functions $P$  
$$\C{\lambda} = \{ P:
\lambda_P = \lambda \mbox{ and $P(v)/\Pi_\lambda(v)$ is log-concave} \}.$$
 Equivalently we require that
$v P(v)^2 \geq (v+1) P(v+1) P(v-1),
\mbox{ for all $v$.}$
\end{condition}
The ULC property is preserved on convolution; that is, for  $P \in \C{\lambda_P}$ and $Q \in \C{\lambda_Q}$ then $P
\star Q \in \C{\lambda_P+\lambda_Q}$ (see \cite[Theorem 1]{walkup} or
\cite[Theorem 2]{liggett}). The ULC class contains a variety of parametric families of random variables, including the Poisson and sums of independent Bernoulli variables
(Poisson--Binomial), including the binomial distribution.

A weaker condition than ULC (Condition \ref{cond:ulc}) is given in \cite{johnson38}, which corresponds to Assumption A of \cite{caputo}:
\begin{definition} Given a probability mass function $P$, write 
\begin{equation} \label{eq:EEdef} \EE(x) := \frac{ P(x)^2 - P(x-1) P(x+1)}{P(x) P(x+1)}
= \frac{P(x)}{P(x+1)} - \frac{P(x-1)}{P(x)}. \end{equation}
\end{definition}
\begin{condition}[$c$-log-concavity] \label{cond:clc}
 If $\EE(x) \geq c$ for all $x \in \Z_+$, we say that $P$ is $c$-log-concave.
\end{condition}
\cite[Section 3.2]{caputo} showed that   if $P$ is ULC then it is $c$-log-concave, with $c = P(0)/P(1)$.
In \cite[Proposition 4.2]{johnson38}  we show that DBE($c$), a discrete form of the Bakry-\'{E}mery condition, is implied by $c$-log-concavity.

Another condition  weaker than the ULC property (Condition \ref{cond:ulc}) was discussed in \cite{johnson24},
related to the definition of a size-biased
distribution. For a mass function $P$, 
we  define its size-biased version $P^*$ by $P^*(x) = (x+1) P(x+1)/\lambda_P$ (note that some authors
define this to be $x P(x)/\lambda_P$; the difference is simply an offset). 
 Now $P \in \C{\lambda_P}$ 
means that $P^*$ is  dominated by $P$ in likelihood ratio ordering (write $P^* \leq_{lr} P$). This is a stronger assumption than  that $P^*$ is
stochastically dominated by $P$ (write $P^* \leq_{st} P$) -- see \cite{shaked} for a review of stochastic ordering results.
As described in \cite{daly,johnson24,papadatos}, such orderings naturally arise in many contexts where $P$ is  the 
mass function of a sum of
negatively dependent random variables.

The other construction at the heart of this paper is the thinning operation introduced by R\'{e}nyi \cite{renyi4}:
\begin{definition} \label{def:thin}
Given probability mass function $P$, define the $\alpha$-thinned version $T_{\alpha} P$ to be
\begin{equation} \label{eq:thin}
T_{\alpha} P(x) = \sum_{y=x}^\infty \binom{y}{x} \alpha^x (1-\alpha)^{y-x} P(y). \end{equation}
Equivalently, if $Y \sim P$ then $T_\alpha P$ is the distribution of a random variable
$T_{\alpha} Y := \sum_{i=1}^{Y} B_{i},$
where $B_{1},B_{2}\ldots$  are IID Bernoulli$(\alpha)$ random variables, independent of $Y$. 
\end{definition}
A key result for Section \ref{sec:maxent} is the fact that (see \cite[Proposition 3.7]{johnson21}) if $P \in \C{\lambda_P}$ then $T_\alpha P \in \C{\alpha \lambda_P}$.
Further properties of this thinning operation are reviewed in  \cite{johnson26}, including the fact that $T_{\alpha}$
preserves several parametric families, such as the Poisson, binomial and negative binomial.
In that paper, we argue that the thinning operation
$T_\alpha$ is the discrete equivalent of scaling by
$\sqrt{\alpha}$ and that
the
`mean-preserving transform' $T_\alpha X + T_{1-\alpha} Y$ is
equivalent to the `variance-preserving transform' $\sqrt{\alpha} X +
\sqrt{1-\alpha} Y$ in the continuous case.
\section{Score function and Poisson approximation}  \label{sec:score}
A well-known phenomenon, often referred to as the `law of small numbers', states that in a triangular array of `small'
random variables, the row sums converge in distribution to the Poisson probability distribution $\Pi_\lambda$. 
For example,  consider the setting where binomial laws $\bin(n, \lambda/n) \rightarrow \Pi_\lambda$ as $n \tends$.
This problem has been extensively studied, with a variety of strong bounds proved using techniques such as the Stein--Chen method
(see for example \cite{barbour} for a summary of this technique and its applications).

Following the pioneering work of Linnik \cite{linnik} applying information theoretic ideas to normal approximation, it was natural to wonder whether  Poisson approximation could be considered in a similar framework. An early and important contribution in this direction was made by
Johnstone and MacGibbon \cite{johnstone}, who made the following natural definition (also used in \cite{kagan,papathanasiou}), which mirrors the Fisher-type expression
of \eqref{eq:lsicont}:
\begin{definition} \label{def:johnstone} Given a probability distribution $P$, define
\begin{equation} \label{eq:johnstone} I(P) := \sum_{x=0}^{\infty} P(x) \left( \frac{P(x-1)}{P(x)} - 1 \right)^2. \end{equation}
\end{definition}
Johnstone and MacGibbon showed that this quantity has a number of desirable features; first 
a Cram\'{e}r-Rao lower bound
$I(P) \geq 1/\var_P$ (see \cite[Lemma 1.2]{johnstone}), with equality if and only if $P$ is Poisson, second a 
projection identity which implies the
subadditivity
result $I( P \star Q) \leq (I(P) + I(Q))/4$ (see \cite[Proposition 2.2]{johnstone}).
 In theory, these two results should allow
us to deduce Poisson convergence in the required manner. However, there is a serious problem. Expanding \eqref{eq:johnstone}
we obtain 
\begin{equation} \label{eq:johnstone2}
I(P) = \left( \sum_{x=0}^\infty \frac{P(x-1)^2}{P(x)} \right) - 1.
\end{equation}
For example, if $P$ is a Bernoulli($p$) distribution then the bracketed sum in \eqref{eq:johnstone2}
 becomes $P(0)^2/P(1) + P(1)^2/P(2)
= (1-p)^2/p + p^2/0 = \infty$. Indeed, for any mass function $P$ with finite support $I(P) = \infty$.
To counter this problem, a new definition was made by Kontoyiannis, Harremo\"{e}s and Johnson in \cite{johnson11}.
\begin{definition}[\cite{johnson11}] \label{def:khj}
For probability mass function $P$ with mean $\lambda_P$, define the scaled score function $\rho_P$ and scaled Fisher information $K$ by
\begin{eqnarray}
\rho_P(x) & := & \frac{ (x+1) P(x+1)}{\lambda_P  P(x)} - 1,  \label{eq:scaledscore} \\
K(P) & := & \lambda_P \sum_{x=0}^\infty P(x) \rho_P(x)^2. \label{eq:scaledfisher}
\end{eqnarray}
\end{definition}
Note that this definition does not suffer from the problems of Definition \ref{def:johnstone} described above. 
\begin{example} \label{ex:bern}
For $P$ Bernoulli($p$), the scaled score function can be evaluated as $\rho_P(0) = P(1)/\lambda_P P(0) - 1 = p/(1-p)$
and $\rho_P(1) = 2 P(2)/\lambda_P P(1) - 1 = - 1$, and the scaled Fisher information as $K(P) = p^2/(1-p)$.
\end{example}
Further, $K$ retains the desirable features of Definition \ref{def:johnstone}; first positivity of the perfect square ensures that
$K(P) \geq 0$ with equality if and only if $P$ is Poisson and second a projection identity \cite[Lemma, P.471]{johnson11} ensures
that the following subadditivity result holds:
\begin{theorem}[\cite{johnson11}, Proposition 3] \label{thm:khj}
 If $S = \sum_{i=1}^n X_i$ is the sum of $n$ independent random variables $X_i$ with means $\lambda_i$ then
writing $\lambda_S = \sum_{i=1}^n \lambda_i$:
\begin{equation} \label{eq:khj}
K(P_{S}) \leq \frac{1}{\lambda_S} \sum_{i=1}^n \lambda_i K(P_{X_i}),
\end{equation}
\end{theorem}
As discussed in \cite{johnson11}, this result  implies Poisson convergence at rates of an optimal order,
 in a variety of senses including
Hellinger distance and total variation. However, perhaps the most interesting consequence comes via the following modified logarithmic
Sobolev inequality of Bobkov and Ledoux \cite{bobkov3}, proved via a tensorization argument:
\begin{theorem}[\cite{bobkov3}, Corollary 4]  \label{thm:bobled}
 For any function $f$ taking positive values:
\begin{equation} \label{eq:bobled}
\Ent{\Pi_\lambda}{f} \leq \lambda \sum_{x=0}^\infty \Pi_{\lambda}(x) \frac{ \left( f(x+1) - f(x) \right)^2}{f(x)}.
\end{equation}
\end{theorem}
Observe that (see \cite[Proposition 2]{johnson11}) taking $f = P/\Pi_{\lambda}$ for some probability mass function $P$, 
the RHS of \eqref{eq:bobled} can be written as 
\begin{equation}
\lambda \sum_{x=0}^\infty \Pi_{\lambda}(x) f(x) \left( \frac{f(x+1)}{f(x)} - 1 \right)^2
= \lambda \sum_{x=0}^\infty P(x) \left( \frac{(x+1) P(x+1)}{\lambda_P P(x)} - 1 \right)^2,
\end{equation}
and so \eqref{eq:bobled} becomes
$ D( P \| \Pi_{\lambda})  \leq  K(P)$. Hence, combining Theorems \ref{thm:khj} and \ref{thm:bobled}, we deduce convergence
in the strong sense of relative entropy in a general setting.
\begin{example} (See \cite[Example 1]{johnson11}) Combining Example \ref{ex:bern} and Theorem \ref{thm:bobled} we deduce
that
\begin{equation} D( B_{n,\lambda/n} \| \Pi_{\lambda}) \leq \frac{\lambda^2}{n(n- \lambda)}, \end{equation}
so via Pinsker's inequality \cite{kullback} we deduce that 
\begin{equation} \| B_{n,\lambda/n} - \Pi_{\lambda} \|_{TV} \leq (2+ \epsilon) \frac{\lambda}{n}, \end{equation}
giving the same order of convergence (if not the optimal constant) as results stated in \cite{barbour}.
\end{example}

We make the following brief remarks, the first of which
 will be useful in the proof of the Poisson maximum entropy property in Section \ref{sec:maxent}:
\begin{remark} \label{rem:ulcequiv}
An equivalent formulation of the ULC property of Condition \ref{cond:ulc} is that $P \in \C{\lambda_P}$ if and only if
 the scaled score function
$\rho_P(x)$ is a decreasing function of $x$. \end{remark}
\begin{remark}
In \cite{johnson22}, definitions of score functions and subadditive inequalities are extended to the compound Poisson setting.
The resulting 
compound Poisson approximation
bounds are close to the strongest known results, due to authors such as Roos 
\cite{roos}.
\end{remark}
\begin{remark} The fact that two competing definitions
 of the score function and Fisher information are possible (see Definition \ref{def:johnstone} and Definition \ref{def:khj})
corresponds to the two definitions of the Stein operator in \cite[Section 1.2]{ley2}. \end{remark}

\section{Poisson maximum entropy property} \label{sec:maxent}

We next  prove a maximum entropy property for the Poisson distribution.
The first results in this direction were proved by Shepp and Olkin \cite{shepp} and by Mateev \cite{mateev}, who showed that
entropy is maximised in the class of Poisson-binomial variables with mean $\lambda$ by the binomial $\bin(n, \lambda/n)$ distribution.
Using a limiting argument, Harremo\"{e}s \cite{harremoes} showed that the Poisson $\Pi_{\lambda}$
 maximises the entropy in the closure of the set of Poisson-binomial variables with mean $\lambda$ (though
note that the Poisson itself does not lie in this set).

Johnson  \cite{johnson21} gave the following argument, based on the idea of `smart paths'. 
For some function $\Lambda(\alpha)$, surprisingly,  it can be easiest to prove $\Lambda(1) \leq \Lambda(0)$ by proving the 
stronger result  that $\Lambda'(\alpha) \leq 0$ for all $\alpha \in [0,1]$.
The key idea in this case
 is to introduce an interpolation between probability measures using the thinning operation of Definition \ref{def:thin}, and to show that it satisfies a partial differential
equation corresponding to the action of the $M/M/\infty$ queue (see also \cite{chafai}).

\begin{lemma} \label{prop:heateqn}
Write $P_\alpha(z) = \pr( T_\alpha X + T_{1-\alpha} Y = z)$,
where $X \sim P_X$ with mean $\lambda$  and $Y \sim \Pi_{\lambda}$
then (in a form reminiscent of \eqref{eq:heateqn} above):
\begin{eqnarray} 
\frac{\partial }{\partial \alpha} P_{\alpha}(x)
& = & \frac{\lambda}{\alpha} \Delta^* (P_{\alpha}(x) \rho_{\alpha}(x)), \label{eq:heateqn2}
\end{eqnarray}
where $\rho_{\alpha} := \rho_{P_\alpha}$ is the score  in the sense of Definition \ref{def:khj} above.
\end{lemma}
The gradient form of the RHS of this partial derivative is key  for us, and is reminiscent of the continuous probability models studied in \cite{ambrosio}.
Such representations will lie at the heart of the analysis of entropy  in the Shepp--Olkin setting \cite{johnson34,johnson36,shepp}
 in Section \ref{sec:so}.

\begin{theorem} \label{thm:maxent}
If $P \in \C{\lambda}$ then
\begin{equation} \label{eq:maxent} H(P) \leq H(\Pi_{\lambda}),\end{equation}
with equality if and only if $P \equiv \Pi_{\lambda}$.
\end{theorem}
\begin{proof}
We consider the functional $\Lambda(\alpha) = -\sum_{x=0}^{\infty}P_\alpha(x) \log \Pi_{\lambda}(x)$ as a function of $\alpha$.
Using Lemma \ref{prop:heateqn} and  the adjoint property of $\Delta$ and $\Delta^*$   we deduce that
\begin{eqnarray}
\Lambda'(\alpha) = - \sum_{x=0}^{\infty} \frac{\partial}{\partial \alpha} P_\alpha(x) \log 
\Pi_{\lambda}(x)  
& = &  - \frac{\lambda}{\alpha}\sum_{x=0}^{\infty} \Delta^* (P_{\alpha}(x) \rho_{\alpha}(x))
 \log \Pi_{\lambda}(x) \nonumber \\
& =  & \frac{\lambda}{\alpha} \sum_{x=0}^{\infty}P_{\alpha}(x) \rho_{\alpha}(x) \log \left( \frac{x+1}{\lambda}\right). \label{eq:derivative}
\end{eqnarray}
Now, the assumption that $P \in \C{\lambda}$ means that $P_\alpha \in \C{\lambda}$ so by Remark \ref{rem:ulcequiv},
the score function $\rho_{\alpha}(z)$ is decreasing in $z$.
Hence \eqref{eq:derivative} is the covariance of an increasing and a decreasing function, so by `Chebyshev's other inequality' (see \cite[Equation (1.7)]{kingman}) it is negative.
 
In other words, $X \in \C{\lambda}$ 
makes $\Lambda'(\alpha) \leq 0$ so $\Lambda(\alpha)$ is  a decreasing function in
$\alpha$.  Since $P_0 = \Pi_{\lambda}$, and $P_1 = P$, we deduce
that
\begin{eqnarray} H(X) & \leq & 
-\sum_{x=0}^{\infty}P(x) \log \Pi_{\lambda}(x)  \label{eq:posrelent} \\ 
& = & \Lambda(1) \leq \Lambda(0) = - \sum_{x=0}^{\infty}\Pi_{\lambda}(x) 
\log \Pi_{\lambda}(x) = H(\Pi_{\lambda}), \nonumber
\end{eqnarray}
and the result is proved.
Here \eqref{eq:posrelent} follows by the positivity of relative entropy.
\end{proof}

\begin{remark}
These ideas were developed by Yu \cite{yu}, who proved that for any $n,\lambda$ 
the binomial  mass function $B_{n,\lambda/n}$ maximises entropy in the class of mass functions $P$ such that $P/B_{n,\lambda/n}$ is log-concave (this class
is referred to as the ULC($n$) class by \cite{pemantle} and \cite{liggett}).
\end{remark}

\begin{remark}
A related extension was given in the compound Poisson case in \cite{johnson23}, one assumption of which was removed by Yu \cite{yu3}. Yu's result \cite[Theorem 3]{yu3} showed that (assuming it is log-concave)
the compound Poisson distribution $CP(\lambda,Q)$ is maximum entropy among all distributions with `claim number distribution' in $\C{\lambda}$ and given `claim size distribution' $Q$.
Yu \cite[Theorem 2]{yu3} also proved a corresponding result for compound binomial distributions.
\end{remark}

\begin{remark}
The assumption in Theorem \ref{thm:maxent} that $P \in \C{\lambda_P}$ was weakened in recent work of Daly \cite[Corollary 2.3]{daly}, who proved the same result assuming that $P^* \leq_{st} P$.\end{remark}

\section{Poincar\'{e} and log-Sobolev inequalities} \label{sec:poincare}
We give a definition which mimics Definition \ref{def:poincont} in the discrete case:
\begin{definition} \label{def:poindisc}
We say that a probability mass function $P$ has Poincar\'{e} constant $R_P$, if for all sufficiently well-behaved functions $g$,
\begin{equation} \label{eq:poindisc} \var_P(g) \leq R_P \sum_{x=0}^{\infty}P(x) \left( \Delta g(x) \right)^2. \end{equation}
\end{definition}
\begin{example} \label{ex:klaassen}
Klaassen's paper \cite{klaasen} (see also \cite{cacoullos})  showed that the Poisson random variable $\Pi_\lambda$ has Poincar\'{e} constant
$\lambda$. As in Example \ref{ex:chernoff}, taking $g(t) = t$, we deduce that $R_P \geq \var_P$ for all $P$ with finite variance.
The result of Klaassen can also be proved by  expanding in Poisson--Charlier polynomials (see \cite{szego}), to mimic
the Hermite polynomial expansion of Example \ref{ex:chernoff}.
\end{example}

However, we would like to generalize and extend Klaassen's work, in the spirit of Theorems 
\ref{thm:contpoin} and \ref{thm:contlsi} (originally due to
Bakry and \'{E}mery \cite{bakry}). Some results in this direction were given by Chafa\"{i} \cite{chafai}, however
we briefly summarize two more recent approaches, using ideas described in two separate papers.

In \cite[Theorem 1.1]{johnson24},
using a construction based on  Klaassen's original paper \cite{klaasen}, it was proved that $R_P \leq \lambda_P$,
assuming that $P^* \leq_{st} P$. Recall (see \cite{shaked}) that this implies $P^*  \leq_{lr} P$, which is a restatement of
the ULC property (Condition \ref{cond:ulc}). We deduce the
following result \cite[Corollary 2.4]{johnson24}:
\begin{theorem} \label{thm:daly}
For $P \in \C{\lambda}$:
$$ \var_P \leq R_P \leq \lambda_P.$$
\end{theorem}
The second approach, appearing in \cite{johnson38}, is based on a birth-and-death Markov chain construction introduced by
Caputo, Dai Pra and Posta \cite{caputo}, generalizing the thinning-based interpolation of Lemma \ref{prop:heateqn}. In
\cite{johnson38}, ideas based on the Bakry-\'{E}mery 
calculus are used to prove the following  two main results, taken from \cite[Theorem 1.5]{johnson38}
and \cite[Theorem 1.3]{johnson38}, which correspond to 
Theorems \ref{thm:contpoin} and \ref{thm:contlsi} respectively.

\begin{theorem}[Poincar\'{e} inequality] \label{thm:poincare}
Any  probability mass function $P$ whose support is the whole of the positive integers $\Z_+$ and which
satisfies the $c$-log-concavity condition  (Condition \ref{cond:clc}) has Poincar\'{e} constant $1/c$.
\end{theorem}
Note that (see the discussion at the end of \cite{johnson38}),  Theorem \ref{thm:poincare} typically gives a weaker result
than Theorem \ref{thm:daly}, under weaker conditions. Next we describe a modified form of log-Sobolev inequality proved in
\cite{johnson38}:
\begin{theorem}[New modified log-Sobolev inequality] \label{thm:lsi}
Fix probability mass function $P$ whose support is the whole of the positive integers $\Z_+$ and which
satisfies the $c$-log-concavity condition  (Condition \ref{cond:clc}). For any function $f$ with
positive values:
\begin{equation} \label{eq:lsi} \Ent{P}{f}\leq \frac{1}{c} \sum_{x=0}^{\infty}P(x) f(x+1)  \left( \log \left( \frac{ f(x+1)}{f(x)} \right) - 1
+ \frac{f(x)}{f(x+1)} \right) . \end{equation}
\end{theorem}

\begin{remark}
As described in \cite{johnson38}, Theorem \ref{thm:lsi} generalizes a modified log-Sobolev inequality of Wu \cite{wu},
which holds in the case where $P = \Pi_\lambda$ (see \cite{yu2} for a more direct proof of this result).
Wu's result in turn strengthens  Theorem \ref{thm:bobled},
the modified logarithmic Sobolev inequality of Bobkov and Ledoux \cite{bobkov3}.
It can be seen directly that Theorem \ref{thm:lsi} tightens Theorem \ref{thm:bobled} using the bound $\log u \leq u-1$ with
$u = f(x+1)/f(x)$.
\end{remark}

Note that since Theorems \ref{thm:poincare} and \ref{thm:lsi} are proved under the $c$-log-concavity condition (Condition
\ref{cond:clc}), they hold under the stronger assumption that $P \in \C{\lambda}$ (Condition \ref{cond:ulc}), taking $c = P(0)/P(1)$.

\begin{remark}
Note that (for reasons similar to those affecting the Johnstone--MacGibbon Fisher information quantity $I(P)$ of Definition
\ref{def:johnstone}), Theorems \ref{thm:poincare} and \ref{thm:lsi} are restricted to mass functions  supported on the whole
of $\Z_+$. In order to consider mass functions such as the Binomial $B_{n,p}$, it would be useful to remove this assumption.

However, one possibility is that in this case, we may wish to modify the form of the derivative used in Definition \ref{def:poindisc}.
In the paper \cite{johnson29}, for mass functions $P$ supported on the finite set $\{0, 1, \ldots, n \}$, a `mixed derivative'
\begin{equation} \label{eq:nablamixed}
\nabla_n f(x) = \left( 1 - \frac{x}{n} \right) \left( f(x+1) - f(x) \right) + \frac{x}{n} \left( f(x) - f(x-1) \right)
\end{equation}
was introduced. This  interpolates between left and right derivatives, according to the position where the derivative
is taken. In \cite{johnson29} it was shown that the Binomial $B_{n,p}$ mass function satisfies a Poincar\'{e} inequality 
with respect to $\nabla_n$. The proof  was based on an expansion  in Krawtchouk polynomials (see \cite{szego}), and 
exactly parallels the results of Chernoff (Example \ref{ex:chernoff}) and Klaassen (Example \ref{ex:klaassen}).
It would be of interest to prove results that mimic Theorem \ref{thm:poincare} and \ref{thm:lsi}
 for the derivative $\nabla_n$ of \eqref{eq:nablamixed}. \end{remark}

\section{Monotonicity of entropy} \label{sec:monotone}

Next we describe results concerning the monotonicity of entropy on summation and thinning, in a regime described in 
\cite{johnson26} as the `law of thin numbers'. 
The first interesting feature is that (unlike the Gaussian case of Theorem \ref{thm:hmonotone}) 
monotonicity of entropy and of 
relative entropy are not equivalent. By convex ordering arguments, Yu \cite{yu2} showed that:
\begin{theorem} \label{thm:yuiid}
Given a random variable $X$ with probability mass function $P_X$ and mean $\lambda$, we
write $D(X)$ for $D( P_X \| \Pi_{\lambda})$.
For independent and identically distributed $X_i$, with mass function $P$ and mean $\lambda$:
\begin{enumerate}
\item
relative entropy
$D \left( \sum_{i=1}^{n} T_{1/n} X_i \right)$ is monotone decreasing in $n$,
\item If $P \in \C{\lambda}$,  the entropy
$H \left( \sum_{i=1}^{n} T_{1/n} X_i \right)$ is monotone increasing in $n$.
\end{enumerate}
\end{theorem}
In fact, implicit in the work of Yu \cite{yu2} is the following stronger theorem: 
\begin{theorem} \label{thm:yugeneral}
Given positive $\alpha_i$ such that $\sum_{i=1}^{n+1} \alpha_i = 1$,
and writing $\alpha^{(j)} = 1- \alpha_j$, then
for any independent  $X_i$,
$$ n D \left( \sum_{i=1}^{n+1} T_{\alpha_i} X_i \right)
\leq \sum_{j=1}^{n+1} \alpha^{(j)} D\left( \sum_{i \neq j} T_{\alpha_i/\alpha^{(j)}} 
X_i \right).$$
\end{theorem}
In \cite[Theorem 3.2]{johnson27}, the following result was proved:
\begin{theorem}
\label{thm:hmon}
Given positive $\alpha_i$ such that $\sum_{i=1}^{n+1} \alpha_i = 1$,
and writing $\alpha^{(j)} = 1-\alpha_j$, then
for any independent ULC $X_i$,
$$ n H \left( \sum_{i=1}^{n+1} T_{\alpha_i} X_i \right)
\geq \sum_{j=1}^{n+1} \alpha^{(j)} H\left( \sum_{i \neq j} T_{\alpha_i/\alpha^{(j)}} 
X_i \right).$$
\end{theorem}
Comparison with \eqref{eq:hmonotone} shows that this is a direct analogue of the result of Artstein et al. \cite{artstein},
replacing differential entropies $h$ with discrete entropy $H$, and replacing scalings by $\sqrt{\alpha}$ by thinnings by
$\alpha$.
Since the result of Artstein et al. is a strong one, and implies strengthened forms of the Entropy Power Inequality, this 
equivalence is good news for us. However, there are two serious drawbacks.

\begin{remark} First, the proof of Theorem \ref{thm:hmon}, which is based on an analysis
of `free energy' terms similar to the $\Lambda(\alpha)$ of Section \ref{sec:maxent},
 does not lend itself to easy generalization or extension. It is possible that a different proof (perhaps using a variational
characterization similar to that of \cite{artstein2,ball}) may lend more insight.
\end{remark}

\begin{remark}
Second, Theorem \ref{thm:hmon} does not lead to a single universally accepted  discrete
Entropy Power Inequality in the way that Theorem \ref{thm:hmonotone} did.
As discussed in \cite{johnson27}, several natural reformulations of the Entropy Power Inequality fail, and while 
\cite[Theorem 2.5]{johnson27} does prove a particular discrete Entropy Power Inequality, it is by no means the only possible
result of this kind.
 Indeed, a growing literature continues to discuss
the right formulation of this inequality.

For example, \cite{shamai,witsenhausen,wyner9,wyner8} consider replacing integer addition by binary addition, with the paper
\cite{wyner8} introducing the so-called Mrs. Gerber's Lemma, extended to a more general
group context by  Anantharam and Jog \cite{anantharam,jog}.
Harremo\"{e}s and Vignat \cite{harremoes3} (see also \cite{sharma}) consider conditions under which the original functional 
form of the Entropy Power Inequality continues to hold.
Other authors use ideas from additive combinatorics and sumset theory, with
Haghighatshoar, Abbe and Telatar \cite{haghighatshoar} adding an explicit error term and Wang, Woo and Madiman \cite{wang9}
giving a reformulation based on rearrangements. It appears that there is considerable scope for extensions and unifications
of all this theory.
\end{remark}

\section{Entropy concavity and the Shepp--Olkin conjecture} \label{sec:so}
In 1981, Shepp and Olkin \cite{shepp} made a conjecture regarding the entropy of Poisson-binomial random variables.
That is, for any vector $\vc{p} := (p_1, \ldots, p_m)$ where $0 \leq p_i \leq 1$, 
we can write $P_{\vc{p}}$ for the probability mass function  of
the random variable $S  := X_1 + \ldots + X_m$, where $X_i$ are independent Bernoulli($p_i$) variables.
Shepp and Olkin conjectured that the entropy of $H( P_{\vc{p}})$ is a concave function of the parameters.
We can simplify this conjecture by considering the affine case, where each $p_i(t) =  (1-t) p_i(0) + t  p_i(1)$.

This result is plausible since Shepp and Olkin \cite{shepp} (see also \cite{mateev})  proved that for any $m$
 the entropy $H(B_{m,p})$ of a Bernoulli random variable is concave in $p$,
and also claimed in their paper that it was true for the sum of $m$ independent Bernoulli random variables when 
$m=2,3$. Since then, progress was limited. In \cite[Theorem 2]{johnsonc6} it was 
proved that the entropy is concave when for each $i$ either $p_i(0) = 0$ or $p_i(1) = 0$. (In fact, this follows from the 
case $n= 2$ of Theorem \ref{thm:hmon} above). Hillion \cite{hillion2} proved the case of the Shepp--Olkin conjecture
where all $p_i(t)$ are either constant or equal to $t$.

However, in recent work of Hillion and Johnson 
\cite{johnson34,johnson36}, the Shepp--Olkin conjecture was proved, in a result stated as
\cite[Theorem 1.2]{johnson36}:
\begin{theorem}[Shepp-Olkin Theorem] \label{th:SO}
For any $m \geq 1$, function $\vc{p}  \mapsto H(P_{\vc{p}})$ is concave.
\end{theorem}
We  briefly summarise the strategy of the proof, and the ideas involved. In \cite{johnson34}, Theorem \ref{th:SO} was
proved in the monotone case where each $p_i' := p_i'(t)$ has the same sign. In \cite{johnson36}, this was extended to the general case; in  fact, 
the monotone case of \cite{johnson34} is shown to be the worst case. In order to bound the second derivative of entropy $\frac{\partial^2}{\partial t^2} H(P_{\vc{p}(t)})$, we need 
to control the first and second derivatives of the probability mass function $P_{\vc{p}(t)}$, in the spirit of Lemma \ref{prop:heateqn}.
Indeed, we write (see \cite[Equations (8),(9)]{johnson36}) the derivative of the mass function in gradient form (see
\eqref{eq:heateqn} and \eqref{eq:heateqn2} for comparison):
\begin{eqnarray}
 \frac{\partial P_{\vc{p}(t)}}{\partial t}(k) & = & g(k-1) - g(k) \\
  \frac{\partial^2 P_{\vc{p}(t)}}{\partial t^2}(k) & = & h(k-2) - 2h(k-1) + h(k) 
\end{eqnarray}
for certain functions $g(k)$ and $h(k)$, about which we can be completely explicit. The key property (referred to as Condition 4 in \cite{johnson34}) is that 
for all $k$:
\begin{eqnarray} \lefteqn{ h(k) \left( f(k+1)^2 - f(k) f(k+2) \right)} \nonumber  \\
& \leq & 2 g(k) g(k+1) f(k+1) - g(k)^2 f(k+2) - g(k+1)^2 f(k), \label{eq:key}
\end{eqnarray}
which allows us to bound the non-linear terms arising in the 2nd derivative $\frac{\partial^2}{\partial t^2} H(P_{\vc{p}(t)})$. The proof of \eqref{eq:key} is based on a 
cubic inequality \cite[Property($m$)]{johnson34} satisfied by Poisson-binomial mass functions, which relates to an iterated log-concavity result of Br\"{a}nd\'{e}n \cite{branden}.

The argument in \cite{johnson34} was based around the idea of optimal transport of probability measures in the discrete
setting, where it was shown that if all the $p_i'$ have the same sign  then the Shepp--Olkin path provides an optimal interpolation, in the sense of
a discrete formula of Benamou--Brenier type \cite{benamou2,benamou} (see also \cite{erbar,gozlan} which also consider optimal transport of discrete measures, including analysis
of discrete curvature in the sense of \cite{caputo}).

\section{Open problems}
We briefly list some open problems, corresponding to each research direction described in 
Sections \ref{sec:score} to \ref{sec:so}. This is not a complete list of open problems in the
area, but gives an indication of some possible future research directions. 

\begin{enumerate}
\item{{\bf [Poisson approximation]}}
The information theoretic approach to Poisson approximation described in Section \ref{sec:score} holds only in the 
case of independent summands. In the original paper \cite{johnson11} bounds are given on the relative entropy, based on the
data processing inequality, which hold for more general dependence structures.
However, these data processing results do not  give Poisson approximation bounds of optimal order.
It is an interesting problem to generalize the projection identity \cite[Lemma, P.471]{johnson11} to the dependent case
(perhaps resulting in an inequality), under some appropriate model of (negative) dependence, allowing Poisson
approximation results to be proved.

Further, although (as described above) such information theoretic approximation results hold for Poisson and compound Poisson limit distributions, it would be of interest
to generalize this theory to a wider class of discrete distributions.

\item{{\bf [Maximum entropy]}}
While not strictly a question to do with discrete random variables, as described in Section \ref{sec:conts} it remains an
interesting problem to find classes within which the stable laws are maximum entropy. As with the ULC class of 
Condition \ref{cond:ulc}, or the fixed variance class within which the Gaussian maximises entropy, we want such a class 
to be well-behaved on convolution and scaling. 
Preliminary work in this direction, including the fact that stable laws are not in general maximum entropy in their 
own domain of normal attraction (in the sense of \cite{gnedenko}), is described in \cite{johnson30}, and an alternative approach based on non-linear diffusion equations and
fractional calculus is given by Toscani \cite{toscani3}.

\item{{\bf [Concentration inequalities]}}
As described at the end of Section \ref{sec:poincare}, it would be of interest to extend these results to the case of
mass functions supported on a finite set $\{ 0, 1, \ldots, n \}$. We would like to introduce concavity
conditions mirroring Conditions \ref{cond:ulc} or \ref{cond:clc}, under which
results of the form Theorem \ref{thm:poincare} and \ref{thm:lsi} can be proved for the derivative $\nabla_n$ of \eqref{eq:nablamixed}.
In addition, it would be of interest to see whether a version of Theorem
\ref{thm:lsi} holds under a stochastic ordering assumption $P^* \leq_{st} P$, in the
spirit of \cite{daly}.

\item{{\bf [Monotonicity]}}
As in Section \ref{sec:monotone}, it is of interest to develop a new and more illuminating
proof of Theorem \ref{thm:hmon}, perhaps based on transport inequalities as in \cite{artstein}. It is possible that such 
a proof will give some indication of the correct formulation of the discrete Entropy Power Inequality, which is itself
a key open problem.

\item{{\bf [Entropy concavity]}} As described in \cite{johnson36}, it is natural to conjecture that the
$q$-R\'{e}nyi \cite{renyi2} and $q$-Tsallis \cite{tsallis} entropies are concave functions for $q$ sufficiently small. Note that while they are monotone
functions of each other, the $q$-R\'{e}nyi and $q$-Tsallis entropies need not be concave for the same $q$.
We quote the following generalized Shepp-Olkin conjecture from \cite[Conjecture 4.2]{johnson36}:
\begin{enumerate}
\item There is a critical 
$q_R^*$ such that the $q$-R\'{e}nyi entropy of all Bernoulli sums is concave for $q \leq q_R^*$, and the entropy of some interpolation is convex for $q > q^*_R$. 
\item There is a  critical 
$q_T^*$ such that the $q$-Tsallis entropy of all Bernoulli sums is concave for $q \leq q_T^*$, and the entropy of some interpolation is convex for $q > q^*_T$. 
\end{enumerate}
Indeed  we conjecture that $q_R^* = 2$ and $q_T^* = 3.65986\ldots$, the root of $2 - 4 q + 2^q = 0$. We remark that the form of discrete Entropy Power Inequality proposed 
by \cite{wang9} (based on the theory of rearrangements) also holds for R\'{e}nyi entropies.

In addition, Shepp and Olkin \cite{shepp} made another conjecture regarding Bernoulli sums, that the entropy $H(P_{\vc{p}})$
is a monotone increasing function in $\vc{p}$ if all $p_i \leq 1/2$, which remains open with essentially no published results.
\end{enumerate}

\section*{Acknowledgments}
The author thanks  
the Institute for Mathematics and Its Applications
for the invitation and funding 
to speak at the workshop `{\em Information Theory and Concentration Phenomena}'  in Minneapolis in April 2015.
In addition, he would like to thank the organisers and participants of this workshop for stimulating discussions.
The author thanks Fraser Daly, Mokshay Madiman and an anonymous referee for helpful comments on earlier drafts of this paper.

\end{document}